\documentclass[10pt,a4paper]{amsart}
\usepackage{graphicx, amssymb, xcolor,pdfsync}
\usepackage[all,cmtip]{xy}
\numberwithin{equation}{section}
\usepackage{mathrsfs}
\usepackage{tikz}
\usepackage{tikz-cd}
\usepackage{faktor}
\usepackage{setspace}
\usepackage{multirow}
\usepackage{array}
\usepackage{comment}
\usepackage{hyperref}
\usepackage{cleveref}
\usepackage{graphicx}
\usepackage{yhmath}
\usepackage{mathdots}
\usepackage{MnSymbol}

\hypersetup{
    colorlinks=true,       
    linkcolor=blue,          
    citecolor=blue,        
    filecolor=blue,      
    urlcolor=blue           
}
\usepackage{tikz}
\usetikzlibrary{matrix,arrows}

\usepackage[switch]{lineno}
\usepackage{yfonts}

\advance\hoffset by -0.9 truecm

\newtheorem{theorem}{Theorem}[section]
\newtheorem{proposition}[theorem]{Proposition}

\newtheorem{lemma}[theorem]{Lemma}

\newtheorem{rema}[theorem]{Remark}

{\theoremstyle{definition}
}

\newcommand{\s}{\vspace{0.2cm}}

\begin{document}
\title[Triangular Riemann surfaces with $2p^2$ automorphisms]{A classification of triangular Riemann surfaces \\ with $2p^2$ automorphisms}
\author{Sebasti\'an Reyes-Carocca}
\address{Departamento de Matem\'aticas, Facultad de Ciencias, Universidad de Chile, Las Palmeras 3425,  Santiago, Chile.}
\email{sebastianreyes.c@uchile.cl}
\author{Yazmin Rivera Nene}
\address{Departamento de Matemáticas, Facultad de Matemáticas, Pontificia Universidad Católica de Chile, Avda. Vicuña Mackenna 4860, Santiago, Chile.}
\email{yrivera1@uc.cl}

\thanks{Partially supported by ANID Fondecyt Regular Grants 1260144 and 1230708, and  ANID Beca Doctorado Nacional  21210737}

\begin{abstract} In this article we provide a classification and description of compact Riemann surfaces admitting a triangular action of a group of order $2p^2,$ where $p$ is an odd prime number. We obtain that all such Riemann surfaces are isomorphic to curves defined over the rational numbers. As a by-product, we derive a classification of orientably-regular hypermaps whose orientation-preserving automorphism group has order $2p^2.$
\end{abstract}
\maketitle
\thispagestyle{empty}
\section{Introduction}

The theory of compact Riemann surfaces  has long been a central research area in pure mathematics, serving to bridge diverse branches such as topology, geometry, algebra, combinatorics and complex analysis.  The origins of this subject can be traced  back to the nineteenth century, with seminal contributions by  Riemann,  Klein and  Hurwitz, among others.

\s

Throughout the development of the area, several significant progress has been achieved; however, two results stand out as fundamental milestones. The first one is the celebrated Riemann-Roch Theorem, which implies that the category of compact Riemann surfaces is equivalent to that of  smooth complex projective algebraic curves.  The second one is the so-called Uniformization Theorem, which implies  that compact Riemann surfaces of genus at least two  are isomorphic to quotients of the complex upper half-plane by the action of discrete groups of automorphisms of it, known as Fuchsian groups.

\s 

Further fundamental results concern automorphisms of compact Riemann surfaces. The finiteness of the automorphism group for genus $g \geqslant 2$ was established by Schwarz in \cite{Sch1879}, and it was Hurwitz in \cite{Hu} who showed that its order is bounded above by $84(g-1)$. Much later, Greenberg in \cite{Greenberg}  proved that every finite group arises as a group of automorphisms of some compact Riemann surface. This result introduced an additional ingredient into the theory, namely, the role played by finite groups and the  significance of their algebraic properties in the geometry of the surfaces on which they act.

\s

Let $\mathscr{M}_g$ denote the moduli space of isomorphism classes of compact Riemann surfaces of genus $g \geqslant 2,$ which is a complex analytic space of dimension $3g-3$. Group actions on compact Riemann surfaces are closely related with the singularities of $\mathscr{M}_g$ in the following sense: for $g \geqslant 4$ it is known that
$$\mbox{Sing}(\mathscr{M}_g)=\{[S] \in \mathscr{M}_g : \mbox{Aut}(S)  \mbox{ is non-trivial}\}.$$Moreover, according to \cite{B90} and \cite{GG92}, the locus consisting of all those Riemann surfaces admitting a {\it fixed type} of group action form an irreducible subvariety of $\mathscr{M}_g.$

\s

Greenberg’s realization theorem, together with the aim of understanding the challenging geometry and topology of $\mathscr{M}_g$ motivates a central problem in the classification of compact Riemann surfaces: {\it determine necessary and sufficient conditions under which a given abstract group can be realized as a group of automorphisms of some Riemann surface under prescribed constraints}. In this direction, several  approaches have been taken, including the following ones.

\s

{\bf a.} {\it Fixing the genus.} Several works address the problem of classifying Riemann surfaces of small genus  (see, for instance,   \cite{BCI45}, \cite{Brou2}, \cite{Marston}, \cite{Kara}, \cite{Ki1} and \cite{Shaska}), while others study group actions in genera of a special form (see, for instance, \cite{BJ}, \cite{IRC}, \cite{May} and \cite{anitayo}). 

\s

{\bf b.} {\it Fixing classes of groups.} Notable contributions include the case of cyclic, abelian and dihedral groups (see \cite{Harvey}, \cite{Macla} and \cite{Buja}). We also refer to \cite{Che}, \cite{Gro3}, \cite{Sch2} and \cite{Z1} for other classes of groups, such as solvable, supersolvable and nilpotent groups.

\s

{\bf c.} {\it Fixing the order of the group}. This case includes the study of Riemann surfaces of genus $g$  with groups of automorphisms of order $8g+8$ and $4g+2$, which yield the celebrated Accola-Maclachlan and Wiman curves; see \cite{K}.  See also \cite{BujalanceCostaIzquierdo} and \cite{CI5} for the cases of $4g$ and $4g+4$.

\s

Beyond the problem of describing the singular locus of the moduli space, the questions above are further motivated by the difficulty of determining the (full) automorphism group of a given Riemann surface or algebraic curve. A powerful tool in this setting, that we will employ in this paper, is the combinatorial  theory of Fuchsian groups and the study of inclusions among them, as developed by Singerman in \cite{Singerman} and \cite{Singerman70}.

\s

Group actions on Riemann surfaces are partially determined by their  signature. The tuple $$(h; m_1, \ldots, m_r) \in \mathbb{Z}^{r+1}  \mbox{ where }  h \geqslant 0 \mbox{ and } m_i \geqslant 2$$is called the {\it signature} of the action of a group $G$ on a compact Riemann surface $S$ if the genus of the orbit space $S/G$ is $h$ and the  branched regular covering map $S \to S/G$ ramifies over exactly $r$ values $y_1, \ldots, y_r$ and the fiber over $y_i$ consists of points with $G$-stabiliser group of order $m_i$ for each $i.$ If $h=0$ and $r=3$ then the action is called {\it triangular} and $S$ is called {\it quasiplatonic}.

\s

The point of view we consider in this article is the one of {\it fixing the type of signature}. More precisely, we consider triangular actions and quasiplatonic Riemann surfaces. The interest in this class of Riemann surfaces arises from several perspectives, some of them we mention below.

\s

{\bf a.} They enjoy the property that cannot be deformed analytically in the moduli space keeping their automorphism group. Roughly speaking, they cannot form families of positive dimension.  

\s

{\bf b.} The geometry of  them is particularly rich, as their analytic structure can be obtained from some bipartite graphs, called  {\it dessins d'enfants} in Grothendieck's terminology \cite{Grot}. 

\s
{\bf c.} By the celebrated Belyi Theorem, they are defined  over number fields (namely, they are isomorphic to curves defined by polynomials with coefficients in $\bar{\mathbb{Q}}$), and hence admit the action of the absolute Galois group.

\s

{\bf d.} Since curves and their Jacobian varieties can be defined over the same fields, these curves provide the natural setting in which to seek Jacobians with complex multiplication.

\s

The study of quasiplatonic Riemann surfaces has been considered by several authors. For instance, the case of low genera was investigated in \cite{low}, while the action of the Galois group  was addressed in \cite{gabinoandrei}. The case of abelian groups was studied in \cite{rubenhom} and \cite{qabelian}, where it was shown that such quasiplatonic Riemann surfaces are defined over the rational numbers. This result was extended in \cite{rubensemi} to certain semidirect products. See also \cite{sobreQ}. We refer the reader to \cite{Wolfart} for an excellent survey on the relationship between these surfaces and Jacobians with complex multiplication. On the other hand, the correspondence between quasiplatonic Riemann surfaces and orientably-regular hypermaps (or dessins d’enfants) has enriched their study through combinatorial methods and has made it possible to exploit advances in computer algebra systems for computer-assisted classifications (see, for instance, \cite{C09} and the references therein). We  refer to  \cite{JS94} and \cite{Wo97} for more details concerning these objects and the relations between them.

\s

A classification of all Riemann surfaces admitting a triangular action of a group of order $pq$, where $p$ and $q$ are prime numbers, was established by Streit and Wolfart in \cite{StreitWolfart}. In the spirit of this classification, for each prime $p \geqslant 3$ we investigate triangular actions of groups of order 
$2p^2$ on Riemann surfaces of genus at least two, and give a complete classification and description of them. The case considered here is more intricate than the $pq$ case, since not only do several distinct groups arise, but each such group may also admit more than one action. 

\s

The results of this article, that will be stated later, can be briefly summarised as follows.

\s

{\bf 1.} We prove that, up to isomorphism, there are precisely 
\begin{displaymath}
 \left\{ \begin{array}{cc}
 8 & \textrm{if  $p=3$}\\
 \frac{p^2+2p+3}{2} & \textrm{if $p \geqslant 5$}
  \end{array} \right.
\end{displaymath}compact Riemann surfaces admitting a triangular  action of a group of order $2p^2$.

\s

{\bf 2.} For each one of these Riemann  surfaces we compute its genus, determine its (full) automorphism group and describe the way this group acts.

\s

{\bf 3.} We observe that these Riemann surfaces are cyclic covers of the projective line, and  provide a description of them as explicit plane affine algebraic curves.  

\s

{\bf 4.} We deduce that each compact Riemann surface admitting a triangular action of a group of order $2p^2$ can be defined over the rational numbers. 

\s

{\bf 5.} We derive a classification of all those orientably-regular hypermaps or dessins d'enfants admitting a group of orientation-preserving automorphisms of order $2p^2.$

\s

Section \ref{preli} is devoted to briefly review the basic preliminaries and  introduce some notation. In Section \ref{s3} we determine the  signatures with which our groups might act. The results will be stated and proved in Sections  \ref{s5}, \ref{s6} and \ref{s4}. We end with a summary of the results in Section \ref{s7}.

\section{Preliminaries and Notations}\label{preli}

\subsection*{Fuchsian groups and Riemann surfaces}
A co-compact {\it Fuchsian group} is a discrete group of automorphisms $\Gamma$ of the upper half-plane $\mathbb{H}$ such that the orbit space $\mathbb{H}/\Gamma$ is compact. The {\it signature} of $\Gamma$ is defined as the tuple
\begin{equation}\label{sigg}s(\Gamma)=(h; m_1, \ldots,m_r),\end{equation}
where $h$ is the genus of $\mathbb{H}/\Gamma$ and the integers $m_1,\ldots,m_r$ are the branch indices in the universal covering map 
$\mathbb{H} \to \mathbb{H}/\Gamma$.  The algebraic structure  of $\Gamma$ is determined by its signature, as   $\Gamma$  has  canonical presentation given by
$$\Gamma= \langle a_1,b_1,\ldots,a_{h},b_{h},x_1, \ldots,x_r  \ | \ \Pi_{i=1}^{h}[a_i,b_i]\Pi_{j=1}^{r}x_j= x_1^{m_1}=\cdots=x_r^{m_r}=1   \rangle,$$
where $[a,b]=aba^{-1}b^{-1}.$ A Fuchsian group (and its signature) is called {\it triangular} if its  signature has the form $(0; m_1, m_2, m_3).$ Such groups  therefore admit the following canonical presentation
$$\Gamma=\Gamma_{(m_1, m_2, m_3)}=\langle x_1,x_2,x_3 \ |  \  x_1x_2x_3=  x_1^{m_1}=x_2^{m_2}=x_3^{m_3}=1 \rangle.$$
Throughout this article, we write $(m_1,m_2,m_3)$ instead of $(0;m_1,m_2,m_3)$.

\s

Fuchsian groups are intimately related with compact Riemann surfaces. Indeed, according to the Uniformization Theorem, if $S$ is  a compact Riemann surface of genus 
$g \geqslant 2$, then there is a Fuchsian group  $\Lambda$ of signature $(g; -)$ such that $S$ and $\mathbb{H}/\Lambda$ are isomorphic. Moreover, the following statements are equivalent:

\s

{\bf (a)} $G$ is a finite group acting on $S \cong \mathbb{H}/\Lambda$ with signature \eqref{sigg}.

{\bf (b)} There is a Fuchsian group $\Gamma$ of signature \eqref{sigg} and a  group epimorphism $\Phi: \Gamma \to G$ such that $\Lambda=\mbox{ker}(\Phi).$

\s

We say that the action of $G$ is {\it represented} by the {\it surface-kernel epimorphism} $\Phi.$ Hereafter, we write {\it ske} for short. Moreover, in such a case, the Riemann–Hurwitz formula can be 
expressed as
$$ 2g-2=|G|(2h-2+ \Sigma_{i=1}^r(1-\tfrac{1}{m_i} ) ).$$ If $G$ acts on $S$ with a triangular signature, then $S$ is called {\it quasiplatonic} and we say that $G$ acts {\it triangularly}. We refer to \cite{Farkas} and \cite{Girondo} for further details on Riemann surfaces and Fuchsian groups. 

\subsection*{Extending  group actions}
Let $S$ be a compact Riemann surface equipped with an action of a finite group $G$ represented by the ske $\Phi : \Lambda \to G$.
Let $G'$ be a finite group such that
$G \leqslant G'$. The action of $G$ on $S$ {\it extends} to an action of $G'$ if the following statements hold.

\s

{\bf (a)} There exists a Fuchsian group $\Gamma'$ with the same dimension that  $\Gamma$, such that $\Gamma \leqslant \Gamma'$.

{\bf (b)} There exists a ske  $\Phi': \Gamma' \to G'$  such that  $\Phi'|_{\Gamma}=\Phi$.

\s

Observe that if the action of  $G$ on $S$ extends, then there is an induced inclusion  of Fuchsian groups. In this way, the lists of all possible inclusions of Fuchsian groups given by Singerman in \cite{Singerman} plays a key role, as it imposes restrictions on the signature of the involved actions. Besides, for each potential extension,   Bujalance, Cirre and Conder in \cite{BCC}  established necessary and sufficient conditions for an action to extend; these conditions are given in terms of the ske representing the action.

\subsection*{Cyclic $n$-gonal actions} Let $n\geqslant 2$ be an integer and let $\omega_n$ be a primitive $n$-th root of unity. A compact 
Riemann surface $S$ is said to be  {\it cyclic 
$n$-gonal} if it admits an automorphism 
$\tau$ of order $n$ such that 
 $S/\langle \tau \rangle \cong \mathbb{P}^1.$ 
The Riemann surface $S$ admits a model as  algebraic curve of the form
\begin{equation}\label{ec-n-gonal}
   \mathscr{C}_S: y^n = \Pi_{i=1}^r (x - \alpha_i)^{n_i}
 \end{equation}
where $n_1, \ldots, n_r$ are positive integers such that
$1 \leqslant  n_i < n$, $n$ divides to $n_1 + \cdots + n_r$ and $\mbox{gcd}(n,n_1, \ldots, n_r) = 1$. In other words, $S$ is isomorphic to the normalisation of $\mathscr{C}_S.$ In addition, in this model the automorphism $\tau$ is represented by the map $(x,y) \mapsto (x,\omega_n y)$ and the values $\alpha_1, \ldots, \alpha_r$ are the branch values of the covering map $\mathscr{C}_S \to \mathbb{P}^1$ which admits $\langle \tau\rangle$ as its deck group. 
The genus of $\mathscr{C}_S$ (and consequently of $S$) is given by 
$$g=\tfrac{1}{2}( 2+ (r-2)n - \Sigma_{i=1}^{r} \mbox{gcd}(n_i,n) ),$$where gcd denotes the greatest common divisor.

Let $P$ be a fixed point of $\tau.$ Then, after considering an appropriate local chart centred in $P$, the automorphism $\tau$ looks locally like $z \mapsto \omega_{n}^kz$ for some $k \in \{1, \ldots, n-1\}.$
 The integer 
$k:=\mbox{rot}(P,\tau)$ is called the \textit{rotation  number} of  $\tau$ at $P$. 
If $k=1$ then  $P$ is called a {\it simple} fixed point of $\tau$.
Observe that  
$$\mbox{rot}(P,\tau^m)=\mbox{rot}(P,\tau)^m \mbox{ for } m\in \Bbb{Z} \,\,  \mbox{ and } \,\, \mbox{rot}(h(P),h\tau h^{-1})=\mbox{rot}(P,\tau) \mbox{ for } h\in \mbox{Aut}(S).$$
Furthermore, if $n=p$ is prime then the integers $n_i$ in \eqref{ec-n-gonal} can be taken as the inverses modulo $p$ of the rotation numbers of $\tau$ at each one of its fixed points. Further details on algebraic models of  
$n$-gonal Riemann surfaces can be found, for example, in
\cite{Broughton}, \cite{Broughton3}  and \cite[Section 2.5]{Girondo}. 
\s

{\bf Notation.} Along this paper, for each integer $n \geqslant 2$ we denote by 
$\mathbb{Z}_n$ and $\mathbb{D}_n$ the cyclic and dihedral groups of order $n$ and $2n,$ respectively. 

\section{Triangular signatures} \label{s3}
 Let $p \geqslant 3$ be a prime number. There are precisely five pairwise non-isomorphic groups of order $2p^2.$ In addition to the cyclic and dihedral ones, these groups are $\Bbb{Z}_{2p}\times \Bbb{Z}_p$, $ 
\Bbb{D}_{p}\times \Bbb{Z}_p$ and the semidirect product
$$\Bbb{Z}_{p}^2\rtimes \Bbb{Z}_{2}=\langle a,b,c \ | \ a^p=b^p=c^2=1, cac=a^{-1}, cbc=b^{-1}, [a,b]=1 \rangle.$$

\begin{lemma}\label{Lema1} Let $g\geqslant 2$ be an integer.

\s
{\bf (1)} $\mathbb{D}_{p^2}$ and $\Bbb{Z}_{p}^2 \rtimes \mathbb{Z}_{2}$ do not act triangularly in genus $g.$

{\bf (2)} If  $\Bbb{Z}_{2p}\times \Bbb{Z}_p$ acts triangularly in genus $g$ then the signature  is $(2p,2p,p)$. 

{\bf (3)} If  $\Bbb{D}_{p}\times \Bbb{Z}_p$ acts triangularly in genus $g$ then the signature is $( 2p,2p,p)$ or $(2,p,2p)$.

  {\bf (4)} If  $\Bbb{Z}_{2p^2}$ acts triangularly in genus $g$ then the signature  is $$(2,p^2, 2p^2), (2p^2, 2p^2, p), (2p^2, 2p^2, p^2) \mbox{ or } (2p, p^2, 2p^2).$$

\end{lemma}
\begin{proof} 

The fact that dihedral groups do not act in genus $g$ triangularly is well-known. Assume that  
$G=\Bbb{Z}_{p}^2\rtimes \Bbb{Z}_{2}$ 
acts triangularly in genus $g$, and let $s=(m_1,m_2,m_3)$ be the 
signature of the action. Then there is a Fuchsian 
group $\Gamma$ of signature $s$ together 
with a ske $\Phi: \Gamma \to G$ representing the action. The surjectivity of $\Phi$ implies that  $G$ must 
be generated by two elements whose orders 
are two or $p.$ Now, the fact that $g \geqslant 2$ implies that the $s$ is either $(2,p,p)$ or $(p,p,p)$ and therefore $G$ can be generated by two elements of order $p.$ However,
two elements of order $p$ generate a 
subgroup of $G$ of order at most $p^2$, 
which yields a contradiction and item (1) follows. Now, assume that 
$G=\Bbb{Z}_{2p}\times \Bbb{Z}_p$ acts triangularly
in genus $g$ with signature $s=(m_1,m_2,m_3)$, and let $\Phi:\Gamma \to G$ be a ske representing such an action.  
Note that $G$ can be 
generated by two elements whose orders lie in the set $\{ 2,p, 2p\}$.
Since $G$ is abelian, it is clear 
that two elements of order $2$ do not generate 
the group. Similarly, neither do two elements of
order $p$, nor a pair consisting of one element 
of order $2$ and one of order $p$. The fact that the product of 
two elements of $G$ of order $2p$ has odd order  shows that   
$s=(2p,2p,p)$. The proof of statement (3) is analogous to the one of statement (2), whereas the last statement follows directly from \cite[Theorem 4]{Harvey}.

\s

\end{proof}


\section{Triangular actions of $\Bbb{Z}_{2p}\times \Bbb{Z}_p$}\label{s5}

Let $p\geqslant 3$ be a prime number. By Lemma \ref{Lema1}, the group $\Bbb{Z}_{2p}\times \Bbb{Z}_p$ can act triangularly in genus $g \geqslant 2$ with signature $(2p,2p,p)$ only.

\begin{theorem}\label{Tma1}
Let $p \geqslant 3$ be a  prime number. There exists a unique compact Riemann surface $X$ of genus $g \geqslant 2$ endowed with a group of automorphisms isomorphic to 
$\Bbb{Z}_{2p}\times \Bbb{Z}_p$ acting  triangularly.  Moreover, the following statements hold.

\s

{\bf (1)} The genus of $X$ is $g=(p-1)^2$.

{\bf (2)} $X$ is isomorphic to the normalisation of the curve given by $y^p=x^{2p}-1$.

{\bf (3)} The automorphism group of $X$ is isomorphic to 
$\Bbb{D}_p\times \Bbb{Z}_{2p}$ and acts with signature $(2p,2p,2)$. 

\end{theorem}  
\begin{proof}
Consider the group $\Bbb{Z}_{2p}\times \Bbb{Z}_p$ with the following presentation $$G=\Bbb{Z}_{2p}\times \Bbb{Z}_p =  \langle a,b,c \ | \ a^p=b^p=c^2= [a,b]=[a,c]=[b,c]=1 \rangle.$$ The non-trivial elements of $G$ are $c$, $a^nb^m$ and $ac^nb^m$ of orders $2$, $p$ and $2p$ respectively, where $0\leqslant n,m\leqslant p-1$ are not simultaneously zero.
Let $\Gamma$ be a Fuchsian group of signature $(2p,2p,p)$ with canonical presentation 
$$\Gamma=\langle x_1,x_2,x_3 \ | \  x_1x_2x_3= x_1^{2p}=x_2^{2p}=x_3^{p}=1   \rangle. $$Note that every ske $\Phi:\Gamma \to G$ representing the action of $G$ has the form 
$$\Phi(x_1,x_2,x_3)=(ca^{n_1}b^{m_1},ca^{n_2}b^{m_2},a^{n_3}b^{m_3}) \ \text{ for some }
\ 0\leqslant n_i,m_i\leqslant p-1,$$ 
where $n_1+n_2+n_3\equiv m_1+m_2+m_3\equiv 0$ mod $p$ and  $n_i$ and $m_i$ are not simultaneously zero. Without loss of generality, 
we can suppose  $n_3\neq 0$ and therefore, modulo the automorphism of $G$ given by 
$a\mapsto a^n, b\mapsto b, c\mapsto c$ where  $nn_3\equiv 1$ mod $p$, we have that $\Phi$ is
$G$-equivalent to 
\begin{equation}\label{G7-Tma1-ec1}
 (ca^{n_1}b^{m_1},ca^{n_2}b^{m_2},ab^{m_3}),   
\end{equation}
with $n_1+n_2+1\equiv m_1+m_2+m_3\equiv 0$ mod 
$p$. After applying the automorphism of 
$G$ defined by 
$a\mapsto ab^{-m_3}, b\mapsto b, c\mapsto c$, we have
that \eqref{G7-Tma1-ec1} is 
$G$-equivalent to $(ca^{n_1}b^{m_1},ca^{n_2}b^{-m_1},a)$, where $n_1+n_2+1\equiv 0$  mod $p$. 
Note that the surjectivity of $\Phi$ ensures that
$m_1\neq 0$ and therefore, by proceeding as before, it can be seen that
 $\Phi$ is
$G$-equivalent to 
$(ca^{n_1}b,ca^{n_2}b^{-1},a)$.
Finally, modulo the automorphism given by  
$a\mapsto a, b\mapsto a^{-n_1}b, c\mapsto c$, 
one sees  that $\Phi$ is $G$-equivalent to
$\Phi_0(x_1,x_2,x_3)=(cb,ca^{-1}b^{-1},a).$ Thus, up to isomorphism, there exists a unique Riemann surface 
$X\cong \Bbb{H}/\mbox{ker}(\Phi_0)$ endowed
 with a triangular action of $\Bbb{Z}_{2p}\times \Bbb{Z}_p$.
    By  the Riemann-Hurwitz formula, the genus of $X$ is $(p-1)^2$. Observe that the curve $y^p=x^{2p}-1$ has the same genus as $X$, and that the maps
$$(x,y) \mapsto (\omega_{2p}x,y) \, \mbox{ and }\, (x,y)\mapsto (x,\omega_py)   $$generate a group of automorphisms of it isomorphic to $G$ acting with signature $(2p,2p,p)$. Thus, the uniqueness ensures that 
$X$ is isomorphic to normalisation of $y^p=x^{2p}-1$.

We proceed to determine the automorphism group of $X.$ Observe that if the automorphism group of $X$ is not isomorphic to $G$ then the covering map $X/G \to X/\mbox{Aut}(X)$ yields an inclusion of $\Gamma$ into a larger  Fuchsian group.
According to the list of maximal signatures given by Singerman in \cite{Singerman}, the signature $(2p,2p,p)$ is not maximal, and we have the following inclusions 
\begin{equation*}\label{G7-Tma1-ec4}
 \Gamma=\Gamma_{(2p,2p,p)} \underset{2}{\trianglelefteq} \Gamma'=\Gamma'_{(2,2p,2p)}\underset{2}{\trianglelefteq} \Gamma''=\Gamma''_{(2,2p,4)} 
\end{equation*}
where $\Gamma''$ is maximal. Furthermore, $\Gamma$ is not contained in any other Fuchsian group. This shows that if $\mbox{Aut}(X)$ is not isomorphic to $G$ then one of the following statements hold.

\s

{\bf (a)} $\mbox{Aut}(X)$ has order $4p^2$ and acts with signature $(2,2p,2p)$.

{\bf (b)} $\mbox{Aut}(X)$ has order $8p^2$ and acts with signature $(2,2p,4)$. 

\s

Observe that the involution $\alpha$ of $G$ defined by
$\alpha(a)=a, \alpha(b)=a^{-1}b^{-1}$ and $\alpha(c)=c$ 
satisfies  
$\alpha(cb)=ca^{-1}b^{-1}.$ 
Then, according to  \cite[Case N8]{BCC}, the action of $G$ on $X$ extends to an action of 
$$G'=\langle A,B,C,W \ | \ A^P=B^P=C^2=W^2=1,\ldots, WAW=A,WBW=A^{-1}B^{-1},WCW=C \rangle$$
with signature $(2,2p,2p)$. Note that $G'= \langle AB^2,W \rangle \times \langle A,C \rangle \cong \Bbb{D}_p\times \Bbb{Z}_{2p}.$
In addition, if we write
$$\Gamma'=\langle y_1,y_2,y_3 \ | \  y_1y_2y_3=1, y_1^{2}=y_2^{2p}=y_3^{2p}=1   \rangle, $$
then the epimorphism $\Phi_0$ induces the ske $$\Phi':\Gamma' \to G' \text{ given by } \Phi'(y_1,y_2,y_3)=(W,CA^{-1}B^{-1},(WCA^{-1}B^{-1})^{-1}),$$
which, according to \cite{BCC}, represents the action of $G'$ on $X$.

Now, we study whether or not the action of $G'$ on $X$ extends. 
Since the surface $X$ is uniquely determined, we may assume the action
of $G'$ to be represented by $(CA^{-1}B^{-1},(W CA^{-1}B^{-1} )^{-1},W)$. Suppose that the action of $G'$ on $X$ extends. It follows from \cite[Case N8]{BCC} that there exists an involution $\beta$ of $G'$ satisfying
\begin{equation}\label{G7-Tma1-ec4-0}
 \beta(CA^{-1}B^{-1})=(WCA^{-1}B^{-1})^{-1} \text{ and }
\beta(W)=AB^{2}W.   
\end{equation}
We claim that  $\beta$ must have the form
\begin{equation}\label{G7-Tma1-ec4-1}
 \beta(A)=A^{n}B^m, \beta(B)=A^{i}B^j \text{ and } \beta(C)=C.
\end{equation} In fact, whereas the first two equalities in \eqref{G7-Tma1-ec4-1} are obvious, 
 the latter equality follows from fact that the elements of order $2$ in $G'$ are of the form $C$,  $WB^tA^k$ and $CWB^tA^k$ with $k$ such that 
$2k\equiv 1$ mod $p$, coupled with the fact that  $WB^tA^k$ and $CWB^tA^k$ 
do not commute with the image of $B$. It follows that \eqref{G7-Tma1-ec4-0} and \eqref{G7-Tma1-ec4-1} imply that $\beta(CA^{-1}B^{-1})=CA^{-n-i}B^{-m-j}=BACW$; a contradiction.  Consequently, the action of $G'$ on $X$ does not extend, and the proof is done.
\end{proof}

\section{Triangular actions of $\Bbb{D}_{p}\times \Bbb{Z}_p$} \label{s6}

 Let  $p\geqslant 3$ be a prime number. According to Lemma \ref{Lema1}, the group $$G=\Bbb{D}_{p}\times \Bbb{Z}_p = \langle a,b,c \ | \ a^p=b^p=c^2=1, cac=a^{-1}, [a,b]=[b,c]=1 \rangle,$$ may act triangularly in genus $g \geqslant 2$  with  signatures $(2,p,2p)$ and $(2p,2p,p)$ only. We study these signatures separately.
For later use, we recall that the non-trivial elements of $G$ are  $ca^k$, $a^nb^m$ and $ca^kb^j$ (of order $2, p$ and $2p$ respectively)  where $0\leqslant k,m,n\leqslant p-1$, $1\leqslant j\leqslant p-1$ and $n$ and $m$ are not simultaneously zero.

\subsection{Actions of signature $(2,p,2p)$}
\begin{theorem}\label{Tma2}
Let $p \geqslant 5$ be a prime number.
There exists a unique  compact Riemann surface $Y$ of genus $g \geqslant 2$ endowed with a group of automorphisms isomorphic to $\Bbb{D}_p\times \Bbb{Z}_p$ acting  with signature $( 2,p,2p)$. 
Moreover, the following statements hold.

\s

{\bf (1)} The genus of $Y$ is $g=\frac{(p-2)(p-1)}{2}$.

{\bf (2)} The automorphism group of $Y$  is isomorphic to the semidirect product $\mathbb{Z}_p^2 \rtimes \mathbb{D}_3$ with presentation 
    \begin{multline*}
\langle A,B,C,W \ | \ A^p=B^p=C^2=W^3=1, \ [A,B]=1, \  WBW^{-1}=A \\
WAW^{-1}=(BA)^{-1}, \ CBC=A, \ CAC=B, \ CWC=W^{-1} \rangle, 
\end{multline*} 
and acts on $Y$ with signature $(2,3,2p)$.

{\bf (3)} $Y$ is isomorphic to the normalisation of the affine algebraic curve $y^p=(x^p-1)^{2}.$

\end{theorem}
\begin{proof}
Consider a Fuchsian group $\Gamma$ of signature 
$(2,p,2p)$  with canonical presentation 
\begin{equation*}\label{prese2.p.2p}\Gamma=\langle x_1,x_2,x_3 \ | \ x_1x_2x_3= x_1^{2}=x_2^{p}=x_3^{2p}=1   \rangle.\end{equation*}
Every action of $G$ with signature $(2,p,2p)$ is represented by a ske
$$\Phi(x_1,x_2,x_3)=(ca^{n_1},a^{n_2}b^{m_2},ca^{n_3}b^{m_3}) \ \text{ for some }
   0\leqslant n_i\leqslant p-1, 1\leqslant m_i\leqslant p-1$$ 
such that $-n_1-n_2+n_3\equiv m_2+m_3\equiv 0$ mod $p$.
By proceeding in a similar way to Theorem \ref{Tma1}, it can be see that $\Phi$ is $G$-equivalent 
to $\Phi_1(x_1,x_2,x_3)=(c,ab^{-1},cab),$ and the uniqueness of $Y$ follows. The genus of $Y$ follows from the Riemann–Hurwitz. According to  \cite{Singerman}, the signature $(2,p,2p)$ is not maximal and admits a unique inclusion (which is non-normal), given by
\begin{equation*}\label{G8-Tma1-ec3}
\Gamma=\Gamma_{(2,p,2p)}\underset{3}{\leq} \Gamma'=\Gamma'_{(2,3,2p)}
\end{equation*}
Thus, the automorphism group of $Y$ is isomorphic either to  $G$ or to a group of order $6p^2$. In the latter case, the automorphism  group acts with signature $(2,3,2p).$ We now proceed to prove that the action of $G$ extends to an action of the group
$G'=\Bbb{Z}_p^2\rtimes \Bbb{D}_3$
 with presentation given in the statement {\bf (2)} of the theorem. Write
$$\Gamma'=\langle y_1,y_2,y_3 \ | \ y_1y_2y_3= y_1^{2}=y_2^{3}=y_3^{2p}=1   \rangle.$$It is straightforward to verify that the group homomorphism
$\Phi':\Gamma'\to G'$  given by   $\Phi'(y_1,y_2,y_3)=(C,WA,CWAB)$
is a ske, showing that $G'$ acts on a compact Riemann surface $Y'$ with signature $(2,3,2p)$. The maximality of $\Gamma'$ guarantees that $\mbox{Aut}(Y')\cong G'.$

\s

{\it Claim.} $Y \cong Y'$ and therefore $\mbox{Aut}(Y)\cong G'$.

\s
Consider the inclusion  
$\eta:\Gamma\to \Gamma'$ given by $\eta(x_1,x_2,x_3)=(y_2y_1y_2^{-1},y_2^{-1}y_3^2y_2,y_3)$ and notice that $\Phi'|_{\Gamma}:\Gamma \to \mbox{Im}(\Phi'|_{\Gamma}) \leqslant G'$ is given by 
$$\Phi'|_{\Gamma}(x_1,x_2,x_3)=(WACA^{-1}W^{-1},A^{-1}W^{-1}(CWAB)^2WA,CWAB).$$ A routine computation shows that $\mbox{Im}(\Phi'|_{\Gamma})=\langle A, B, CW \rangle=
\langle A, B, CW \rangle \cong  G= \Bbb{D}_p\times \Bbb{Z}_p,$ and hence $G$ acts on $Y'$ with signature $(2,p,2p)$. The claim follows from the uniqueness of $Y$. 

\s

It only remains to give an algebraic description of $Y.$ To accomplish this task, consider  the regular covering map 
$\pi : Y \to Y/G$ induced by the action of $G$ on $Y$,
and let $a_1, a_2$ and $a_3$ be its ordered branch values. 
Let $N \cong \mathbb{Z}_p$ be the normal subgroup of $G$ generated by $b$ and 
denote by $\pi_N : Y \to Y/N$ the regular covering map 
induced by the action of $N$ on $Y$. The $N$-stabilisers of the branch points of $\pi$ are trivial, except for those forming the $\pi$-fibre of $a_3$ (they are $p$  points with $G$-stabiliser conjugate to
$\langle cab \rangle=\langle ca \rangle \times \langle b\rangle$). 
Thus,  $\pi_N$  has precisely  $p$ branch values of all of 
them marked with $p$. By the Riemann–Hurwitz 
formula, the quotient $Y/N$ has genus zero, showing that $Y$ is  cyclic $p$-gonal. Thus, $Y$ admits an algebraic model  
given  by the curve
\begin{equation}\label{G8-Tma1-ec7}
    y^{p}=\Pi_{i=1}^{p}(x-\alpha_i)^{n_i},
\end{equation}
where $\alpha_1,\ldots, \alpha_p$ are the branch values of
$\pi_N$ and $n_1, \ldots, n_p$ are integers in $\{1, \ldots, p-1\}.$ Let $\hat{a}$ be the automorphism of $Y/N$ induced by $a.$ The action of $\langle \hat{a} \rangle  \cong  \Bbb{Z}_p$ has
two fixed points that, up to a Möbius transformation, we  can assume to be $\infty$ and $0$. Thus, 
$\hat{a}(z) = \omega_p z$ and  the $p$ branch 
values of $\pi_N$ can be  assumed 
to be the $p$-th roots of unity.  Thus, \eqref{G8-Tma1-ec7} turns into
\begin{equation}\label{G8-Tma1-ec8}
    y^p=(x-1)^{n_1}(x-\omega_p)^{n_2}\cdots(x-\omega_p^{p-1})^{n_p}.
\end{equation}
To determine the values $n_1, \ldots, n_p$,  we make use of the rotation numbers. 
We recall that there is $P \in \pi^{-1}(a_3)$ such that $P$ is a primitive simple fixed point of $cab$, namely,
$\mbox{rot}(P,cab)=\omega_{2p}$. This implies that $\mbox{rot}(P,b^2)=\omega_{p}$ and  $\mbox{rot}(P,b)=\omega_{p}^{\beta}$ where $2\beta \equiv 1$  mod $p$. Now, if $Q \in \pi^{-1}(a_3)$ then $Q=g(P)$ for some $g \in G,$ and therefore  
$$\mbox{rot}(Q,b)=\mbox{rot}(g(P),b)=\mbox{rot}(P,g^{-1}bg)=\mbox{rot}(P,b)=\omega_p^{\beta}.$$
We then conclude that $n_i= 2$ for each $1 \leqslant i \leqslant p$ and the proof is done.
\end{proof}

\subsection{Actions of signature $(2p,2p,p)$}

Let $\Gamma$ be a Fuchsian group of signature $(2p,2p,p)$ with canonical presentation
\begin{equation*} \label{prese.2p.2p.p}\Gamma=\langle x_1,x_2,x_3 \ | \  x_1x_2x_3= 
x_1^{2p}=x_2^{2p}=x_3^{p}=1 \rangle.\end{equation*}
Consider the ske
$$\Phi_{j}: \Gamma\to G \ \text{given by} \ \Phi_{j}(x_1,x_2,x_3)= (cb, ca^{-1}b^{-1-j},ab^{j})$$
where $0\leqslant j\leqslant p-2$, and
 denote by $Z_j$ the Riemann surface defined by $\Phi_j$.

\begin{theorem}\label{Tma3}
Let $p \geqslant 3$ be a prime number. There are exactly $\frac{p+1}{2}$ pairwise non-isomorphic compact Riemann surfaces of genus $g \geqslant 2$ admitting a group of automorphisms isomorphic to   $\Bbb{D}_p\times \Bbb{Z}_p$ 
acting with signature $(2p,2p,p)$. Moreover, the following statements hold.  

\s

{\bf (1)} These Riemann surfaces are $Z_0, \ldots, Z_{p-2}$ up to the following identification: for $i \neq j \in \{1, \ldots, p-3\}$
$$Z_i \cong Z_j \mbox{ if and only if } i+ij+j \equiv 0 \mbox{ mod }p.$$

 {\bf (2)} The genus of each $Z_j$ is $g=(p-1)^2$.

 {\bf (3)} The automorphism group of $Z_j$ is isomorphic to 
 \begin{displaymath}
 \left\{ \begin{array}{cc}
\Bbb{D}_{p}^2\rtimes \Bbb{Z}_2 & \textrm{if  $j=0$}\\
 \Bbb{D}_p\times \Bbb{Z}_{2p} & \textrm{if $j=p-2$}\\
 \Bbb{D}_p\times \Bbb{Z}_p & \textrm{otherwise}
  \end{array} \right.
\end{displaymath}where
the former semidirect product  has presentation 
 
\begin{multline*}
\langle A,B,C,W,U \ | \ A^p=B^p=C^2=W^2=U^2=1, \ [A,B]=[A,W]=[B,C]=[C,W]=1, \\ CAC=A^{-1}, WBW=B^{-1}, UAU=B^{-1},UBU=A^{-1}, UCU=WB^2,UWU=CA^2 \rangle .  
\end{multline*}The signature of the action of the automorphism group is
 $(2,2p,4)$ if $j=0$ and $(2p,2p,2)$ if $j=p-2.$

{\bf (4)} $Z_{j}$ is isomorphic to the normalisation of the affine algebraic curve
$$y^p=(x^p-1)^2(x^p+1)^{\epsilon_j p-2(j+1)},$$
where $\epsilon_j$ is an integer chosen in such a way that $1\leqslant \epsilon_j p-2(j+1) < p$.
\end{theorem} 

\begin{proof}
Let $Z$ be a compact Riemann surface endowed with a group of automorphisms isomorphic to $G=\Bbb{D}_p\times \Bbb{Z}_p=\langle a,c \rangle \times \langle b \rangle$
acting with signature  $(2p,2p,p)$. Note that each ske 
$\Phi:\Gamma \to G$ that represents an action of $G$ with signature $(2p,2p,p)$ is given by 
$$\Phi(x_1,x_2,x_3)=(ca^{n_1}b^{m_1},ca^{n_2}b^{m_2},a^{n_3}b^{m_3}) \text{ for some } 0\leqslant n_i, m_3\leqslant p-1, 1\leqslant m_1,m_2\leqslant p-1$$
that satisfies $-n_1+n_2+n_3\equiv m_1+m_2+m_3\equiv 0$ mod $p,$ where $n_3$ and  $m_3$ are not simultaneously zero. 
By proceeding analogously to Theorem \ref{Tma1}, it can be seen that $\Phi$ is $G$-equivalent to 
\begin{equation}\label{G8-Tma2-ec0}
\Phi_j(x_1,x_2,x_3)=(cb,ca^{-1}b^{-1-j},ab^{j}) \text{ for some } 0\leqslant j\leqslant p-2.
\end{equation}
Hence, $Z$ is isomorphic to $Z_j$ for some $0\leqslant j\leqslant p-2$. It is not difficult to see that if $i \neq j$ then $\Phi_i$ and $\Phi_{j}$ are $G$-inequivalent. Note that the genus of $Z_j$ is $(p-1)^2$. 

\s

The signature $(2p,2p,p)$ is not maximal and  
the following chain of inclusions  holds.
\begin{equation*}\label{G8-Tma2-ec1}
   \Gamma=  \Gamma_{(2p,2p,p)} \underset{2}{\trianglelefteq} \Gamma'=\Gamma'_{(2p,2p,2)}\underset{2}{\trianglelefteq} \Gamma''=\Gamma''_{(2,2p,4)},
\end{equation*}
where $\Gamma''$ is maximal and $\Gamma$ is not contained in any other Fuchsian group \cite{Singerman}. Hence, there are three possibilities:

\s

{\bf (a)} $\mbox{Aut}(Z_j) \cong G,$

 {\bf (b)} $\mbox{Aut}(Z_j)$ has order $4p^2$ and acts on $Z_j$ with signature $(2p,2p,2),$ and

{\bf (c)} $\mbox{Aut}(Z_j)$ has order $8p^2$ and acts on $Z_j$ with signature $(2,2p,4).$

\s

Assume that the action of $G$ on $Z_j$ extends to a larger group $G'$ of order $4p^2.$ The classification of groups of order $4p^2$ is known (see, for instance, \cite[Section IV]{Joel}). After restricting to those groups that contain a subgroup isomorphic to $G$
 and requiring that they admit an action of signature $(2,2p,2p)$, one obtains that $G'$ is  isomorphic to either $G_1=  \Bbb{D}_{p}\times \Bbb{D}_p$ or $G_2= \Bbb{D}_p\times \Bbb{Z}_{2p}$.

\s

\it{Claim 1:}\rm \ The action of $G$ on $Z_j$ extends to an action of $G_1$ if and only if $j=0$.

\s

We consider the canonical presentation of $\Gamma'$  \begin{equation}\label{felix}\Gamma'=\langle y_1, y_2, y_3 \ | \  y_1y_2y_3= y_1^{2p}=y_2^{2p}=y_3^{2}=1 \rangle\end{equation}and the following presentation of $G_1$.$$G_1 = \langle A,B,C,W \ | \ A^p=B^p=C^2=W^2=1,  \ [A,B]=[A,W]=[B,C]=[C,W]=1, CAC=A^{-1}, WBW=B^{-1} \rangle.$$The involutions of $G_1$ are $CA^k$, $WB^j$ and $CA^kWB^j$, whereas the elements of order $2p$ of $G_1$ are of two types: the first type is given by $CA^kB^{l}$ and the second type by $WB^jA^{i}$, where $0\leqslant k,j\leqslant p-1$, $1\leqslant l,i \leqslant p-1$. The fact that two elements of order $2p$ of the same type  cannot generate $G_1$ implies that  each ske $\Gamma'\to G_1$ representing an action of $G_1$ with signature 
$(2p,2p,2)$ is given by 
\begin{equation*}\label{Ec-Tma3-1}
(CA^{k_1}B^{j_1},WB^{j_2}A^{k_2},CA^{k_3}WB^{j_3}) \ \text{or} \ 
(WB^{j_4}A^{k_4},CA^{k_5}B^{j_5},CA^{k_6}WB^{j_6})
\end{equation*}
 where $ 0\leqslant k_1,k_3,k_5,k_6,j_2,j_3,j_4,j_6 \leqslant p-1  \text{ and } 1\leqslant k_2,k_4,j_1,j_5 \leqslant p-1 $
satisfy that $-k_1-k_2+k_3 \equiv j_1-j_2+j_3 \equiv 
-j_4-j_5+j_6 \equiv k_4-k_5+k_6 \equiv 0$ mod $p$. 
A routine computation shows that all these skes are pairwise $G_1$-equivalent and therefore $G_1$ acts on a unique compact Riemann surface $Z'$ with signature $(2p,2p,2)$. Without loss of generality, we can assume such an action to be represented by the ske $\Phi_1' : \Gamma' \to G_1$ given by  $$\Phi_1'(y_1, y_2, y_3)=(WBA^{\beta}, CB, CA^{-\beta}W) \mbox{  where } 2 \beta \equiv 1 \mbox{ mod }p.$$

We consider the inclusion of Fuchsian groups 
$\eta: \Gamma \to \Gamma'$ given by $\eta(x_1,x_2,x_3)=(y_2,y_3y_2y_3^{-1},y_1^2)$. The restriction of $\Phi_1'$ to $\eta(\Gamma)\cong \Gamma$ is given by 
\begin{equation}\label{G8-Tma2-ec3}
    \Phi_1'|_{\Gamma}(x_1,x_2,x_3)=\Phi_1'(y_2,y_3y_2y_3^{-1},y_1^2)= (CB, CA^{-1}B^{-1}, A).
\end{equation}Note that $\Phi_1'(\Gamma)=\langle A, B, C \rangle \cong \mathbb{D}_p \times \mathbb{Z}_p$ and therefore $Z' \cong Z_j$ for some $0 \leqslant j \leqslant p-2.$ The proof Claim 1 follows after noticing that \eqref{G8-Tma2-ec3}
agrees with 
\eqref{G8-Tma2-ec0} with $j=0$.

\s

\it{Claim 2:}\rm \ $\mbox{Aut}(Z_0)\cong G_1'$, where  $G_1'\cong \Bbb{D}_p^2\rtimes \Bbb{Z}_2$ is the group of order $8p^2$ with   presentation 
\begin{multline*}
 \langle A,B,C,W,U \ | \  A^p=B^p=C^2=W^2=U^2=1, \ [A,B]=[A,W]=[B,C]=[C,W]=1, \\ CAC=A^{-1}, WBW=B^{-1}, UAU=B^{-1},UBU=A^{-1}, UCU=WB^2,UWU=CA^2 \rangle.   
\end{multline*}

The map given by
$A \mapsto B^{-1}, B \mapsto A^{-1}, C \mapsto WB^{2}$ and $W \mapsto CA^{2}$ defines an automorphism of $G_1$ of order two that satisfies the equalities \begin{equation}\label{G8-Tma2-ec4-2}
    \alpha(CAB)=WBA^{-1}  \ \text{ and } \ 
     \alpha(CW)=CA^2WB^2.
\end{equation}It follows from \cite[Case N8]{BCC} that the action of $G_1$ on $Z_0$ extends to an action of $G_1'$ with signature $(2,2p,4)$. The maximality of the signature $(2,2p,4)$ implies that
$\mbox{Aut}(Z_0)\cong G_1'$ and the proof of Claim 2 is done.

\s

\it{Claim 3:}\rm \ The action of $G$ on $Z_{j}$  extends to an action of $G_2$ if and only if $j=p-2$.

\s

We consider $G_2$ with the following presentation
$$\langle A,B,C,W \ |  A^p=B^p=C^2=W^2=1, CAC=A^{-1}, [A,B]=[A,W]=[B,W]=[B,C]=[C,W]=1 \rangle.$$
The elements of order $2p$ in $G_2$ are of three type: the ones of type 1 are $CA^kB^j$, the ones of type 2 are
$WA^kB^i$ and the ones of type 3 are $CA^kWB^j$, where $0\leqslant k,i \leqslant p-1, 1\leqslant j\leqslant p-1,$ with $k$ and $i$ are  not simultaneously zero. The involutions of $G_2$ are $W$, $CA^k$ and  $CA^kW$ for $0\leqslant k \leqslant p-1.$  

\s
We consider the Fuchsian group $\Gamma'$ with canonical presentation \eqref{felix}. If $$\Phi_2':\Gamma'\to G_2 \, \mbox{ given by } \Phi'_2(y_1,y_2,y_3)=(g_1,g_2,g_3)$$ is a ske representing an action of $G_2$ of signature $(2p,2p,2)$ then
$g_1$ and $g_2$ cannot be of the same type. Indeed,  if this were the case, then product $g_1g_2$ would not have order $2$.
Moreover, $g_1$ and $g_2$ cannot be of types $1$ and $3$ respectively (nor of types 3 and 1, respectively) because  otherwise the requirement that the product $g_1g_2$ have order two would force $g_3=W.$ This, however, would imply that $\Phi'_2$ is not surjective. It follows that, after considering the automorphism of $G_2$ given by $A\mapsto A, C\mapsto CW, B\mapsto B, W\mapsto W$, the ske
$\Phi_2'$  is given by
\begin{equation}\label{G8-Tma2-ec6-4}
(CA^{k_1}WB^{j_1},WA^{k_2}B^{j_2}, CA^{k_3})  
\text{ where }
0\leqslant k_1,k_2,k_3,j_2\leqslant p-1,  \ 1\leqslant j_1\leqslant p-1 
\end{equation} 
or by
\begin{equation}\label{G8-Tma2-ec6-4-2}  
(WA^{k_4}B^{j_4},CA^{k_5}WB^{j_5},CA^{k_6}) \text{ where }
0\leqslant k_4,k_5,k_6,j_4\leqslant p-1,  \ 1\leqslant j_5\leqslant p-1,
\end{equation}
where $k_2,j_2$ and $k_4,j_4$ are not simultaneously zero, and $-k_1-k_2+k_3\equiv j_1+j_2\equiv k_4-k_5+k_6\equiv j_4+j_5\equiv 0$ mod $p$. It is a routine computation to show that 
 \eqref{G8-Tma2-ec6-4} and \eqref{G8-Tma2-ec6-4-2} are $G_2$-equivalent to$$\Phi'_{2,1}(y_1,y_2,y_3)=(CAB,WA^{-1}B^{-1}, CW) \mbox{ and }
\Phi'_{2,2}(y_1,y_2,y_3)=(WAB,CAB^{-1},CW)$$ respectively.
We then conclude that $G_2$ acts on two Riemann surfaces with signature $(2p,2p,2)$, that we denote by $Z_i'=\mathbb{H}/\mbox{ker}(\Phi'_{2,i})$ for $i=1,2.$ Let $$\eta_1(x_1, x_2, x_3) = (y_3 y_1 y_3^{-1}, y_1, y_2^2) \mbox{ and }\eta_2(x_1, x_2, x_3) = (y_2,y_3y_2y_3^{-1},y_1^2)$$ be two inclusions of $\Gamma$ in $\Gamma'.$ By proceeding as in the proof of Claim 1, we obtain that 
 the restriction of $\Phi'_{2,i}$ to $\Gamma \cong \eta_i(\Gamma)$ is a ske $\Gamma \to G$ which is, in both cases, $G$-equivalent to \eqref{G8-Tma2-ec0} with $j=p-2$. As a result, we obtain that $Z_1'\cong Z_2'\cong Z_{p-2}$ and the proof of Claim 3 follows. 

\s

\it{Claim 4:}\rm \ $\mbox{Aut}(Z_{p-2}) \cong \Bbb{D}_p\times \Bbb{Z}_{2p}.$

\s

As mentioned at the beginning of the proof, if $\mbox{Aut}(Z_{p-2})$ is not isomorphic to $\Bbb{D}_p\times \Bbb{Z}_{2p}$ then $\mbox{Aut}(Z_{p-2})$ has order $8p^2$ and acts on $Z_{p-2}$ with signature $(2,2p,4).$ We consider again the results in  \cite{BCC} to prove that the aforementioned situation is not realized. To do that, it is enough to verify that there is no involution   $\alpha$ of $G_2$
satisfying the following relations:
\begin{equation}\label{G8-Tma2-ec10}
    \alpha(CAB)=WA^{-1}B^{-1}\ \text{ and } \
    \alpha(CW)=CA^2W.
\end{equation}
Note that the first equality above implies that $\alpha(W)=CA$. This, coupled with the second equality above, forces  $$\alpha(C)\alpha(W)=\alpha(C)CA=CA^2W \, \mbox{ and therefore } \, \alpha(C)=A^{-1}W;$$which is impossible. This proves Claim 4. 

\s

Claims 1 and 3 imply that
$\mbox{Aut}(Z_j)\cong G$  for each  
$1\leqslant j\leqslant p-3.$ We now observe that these Riemann surfaces are isomorphic in pairs, yielding $\tfrac{p-3}{2}$ isomorphism classes.

\s

\it{Claim 5:}\rm \ Let $i \neq j$ be integers in $\{1, \ldots, p-3\}$. Then $Z_j \cong Z_i$ if and only if $i+ij+j\equiv 0$ mod $p.$ 

\s
We recall that the action of $G$ on each $Z_j$ is represented by the ske
$$\Phi_j(x_1,x_2,x_3)=(cb,ca^{-1}b^{-1-j},ab^{j}).$$ The fact that $G\cong \mbox{Aut}(Z_j)$ implies that 
$\Gamma=N(K_j)$ where $K_j=\mbox{ker}(\Phi_j)$, and therefore 
$Z_i$ and $Z_j$ are isomorphic if and only if the corresponding surface groups $K_j$ and $K_i$
 are conjugate in $N(\Gamma)=\Gamma'_{(2,2p,2p)}$.
 Since the stabilizer of each surface group under the action by conjugation  of $\Gamma'$ is $\Gamma$, we obtain that the action by conjugation has orbits of length
 $[\Gamma':\Gamma]=2$.
If we consider the inclusion $\Gamma \to \Gamma'$ given by
$$(x_1,x_2,x_3) \mapsto (z_1, z_2, z_3):=(y_2^{-1}y_1y_2,y_1,y_2^2)$$ then $\Gamma'/\Gamma = \langle Y_2 \ | \ Y_2^2=1 \rangle $, where $Y_2$ stands for the class of $y_2$. The action by conjugation of $\Gamma'/\Gamma$ is given by
$$Y_2. z_1= z_2, \  Y_2 . z_2 =z_2^{-1}z_3^{-1}  \ \text{ and } \  Y_2 . z_3 = z_3,$$ showing that $\Phi_j=(cb,ca^{-1}b^{-1-j},ab^{j})$ and $Y_2\cdot \Phi_j=(ca^{-1}b^{-1-j}, ca^{-2}b , ab^{j})$  form an orbit under this action.
To prove Claim 5 it enough to observe that, after applying the automorphism of $G$ given by 
 $a \mapsto a,b\mapsto b^{\delta}, c\mapsto ca$ where $\delta$ satisfies $\delta(1+j)\equiv -1$ mod $p$, the ske  $Y_2 \cdot \Phi_j$ is $G$-equivalent to 
$$\Phi_{-\delta-1}(x_1,x_2,x_3) =
(cb, ca^{-1}b^{\delta} , ab^{-1-\delta}).$$If we write $i=-\delta-1$ then $i+ij+j \equiv 0 \mbox{ mod }p,$ as claimed.

\s

All the above shows that there are precisely   $\frac{p-3}{2}+2=\frac{p+1}{2}$  pairwise 
non‑isomorphic Riemann surfaces admitting an action of $G=\mathbb{D}_p \times \mathbb{Z}_p$ with signature $(2p,2p,p)$.

\s

In order to obtain the desired algebraic model for each $Z_j$ we consider the regular covering map $\pi_j:Z_j\to Z_j/G$  induced by the action of $G$ on $Z_j$, and denote by $a_1,a_2,a_3$ its the ordered  branch values.
Let $N \cong \mathbb{Z}_p$ the normal subgroup of $G$ generated by $b$ and let  $\pi_{j,N}:Z_j\to Z_j/N$ be the corresponding regular covering map. A routine computation shows that $\pi_{j,N}$ ramifies over $2p$ values and $Z_j$ is $p$-gonal. Thus, $Z_j$  admits a model as the curve   
\begin{equation}\label{G8-Tma2-ec20}
    y^p=\Pi^{2p}_{i=1}(x-\alpha_i)^{n_i},
\end{equation}
where $\alpha_1,\ldots,\alpha_{2p}$ are the branch values of  $\pi_{jN}$ and $n_1,\ldots,n_{2p}$ are integers in
$\{1,\ldots, p-1 \}$. The fact that $Z_j/N$ admits an action of $K=\mathbb{D}_p$ such that $(Z_j/N)/K\cong Z_j/G$ allows us to assume that $\alpha_i=\omega_p^{i-1}$ and $\alpha_{p+i}=\lambda\omega_p^{i-1}$ for $i=1, \ldots, p$, where $\lambda^p=-1$. Finally, we argue as done in the proof of statement (3) in Theorem \ref{Tma2} to see that $n_i=2$ for $i=1, \ldots, p$ and $n_i=\epsilon_jp-2(j+1)$ for $i=p+1,\ldots, 2p$, where $\epsilon_j$ an integer chosen in such a way that $1 \leqslant \epsilon_j p -2(j+1)\leqslant p-1$. 
Hence, equation \eqref{G8-Tma2-ec20} becomes the desired equation, and this completes the proof.
\end{proof}

\begin{rema} It follows from the proof of Claim 3 above that
the Riemann surfaces $X$ of Theorem \ref{Tma1} and  $Z_{p-2}$ of Theorem \ref{Tma3} are isomorphic.
\end{rema}

\section{The cyclic case}\label{s4}

Let $p \geqslant 3$ be a prime number. We recall that, according to Lemma \ref{Lema1}, the cyclic group of order $2p^2$ can act triangularly in genus $g \geqslant 2$ with signatures $$(2,p^2, 2p^2), (2p^2, 2p^2, p), (2p^2, 2p^2, p^2) \mbox{ or } (2p, p^2, 2p^2)$$only. Indeed, by \cite[Theorem 4]{Harvey} each such a signature is realised.

\begin{theorem}\label{Tma4}
Let $p \geqslant 3$ be a prime number. Let $S$ be a compact Riemann surface of genus $g \geqslant 2$ admitting a group of automorphisms $G$ isomorphic to the cyclic group of order $2p^2$ acting triangularly.  Then one of the following statements holds.

\s

{\bf (1)} $g=\tfrac{1}{2}(p^2-1)$ and $S$ is isomorphic to the normalisation of the curve $\mathscr{X}_p$ given by $y^{p^2}=x^2+1.$ The automorphism group of $S$ is isomorphic to $G$ and acts with signature $(2, p^2, 2p^2).$

\s

{\bf (2)} $g=p(p-1)$ and $S$ is isomorphic to the normalisation of the curve $\mathscr{Y}_{p,m}$ given by $$y^p=x^2(x^{2p}-1)^m \, \mbox{ for some }  m \in \{1, \ldots, \tfrac{p-1}{2}\}.$$ The automorphism group of $S$ is isomorphic to $G$ and acts  with signature $(2p^2, 2p^2, p).$ Moreover $\mathscr{Y}_{p,m} \cong \mathscr{Y}_{p,m'}$ if and only if $m=m',$ and there are exactly $\tfrac{p-1}{2}$ isomorphism classes of Riemann surfaces $S$ in this case.

\s

{\bf (3)} $g=\tfrac{(p-1)(2p+1)}{2}$ and $S$ is isomorphic to the normalisation of the  curve $\mathscr{W}_{p,k}$  given by  $$y^{p^2}=x^{kp}(x^{2}-1) \mbox{ for some } k \in \{1, \ldots, p-1\}.$$The  automorphism group  is isomorphic to $G$ and acts  with signature  $(2p,p^2, 2p^2).$ Moreover
$\mathscr{W}_{p,k} \cong \mathscr{W}_{p,k'}$ if and only if $k=k',$ and there are exactly $p-1$ isomorphism classes of Riemann surfaces $S$ in this case.

\s

{\bf (4)} $g=p^2-1$  and $S$ is isomorphic to the normalisation of the  curve $\mathscr{Z}_{p,l}$ given by $$ y^{p^2}=x^2(x^{2p}-1)^l \mbox{ for some } l \in \{1, \ldots, p^2\} \mbox{ such that } (l,p)=(l-1, p)=1.$$

The automorphism group of $S$ is isomorphic to \begin{displaymath}
 \left\{ \begin{array}{cc}
 \mathbb{Z}_{p^2} \rtimes \mathbb{D}_4 & \textrm{if  $S \cong \mathscr{Z}_{p,p-2}$}\\
 G & \textrm{otherwise}
  \end{array} \right.
\end{displaymath}where $\mathbb{Z}_{p^2} \rtimes \mathbb{D}_4$ stands for the group with presentation \begin{equation}\label{azul}\langle A, R, T \, | \, A^{p^2}=R^4=T^2=(TR)^2=1, RAR^{-1}=TAT=A^{-1} \rangle.\end{equation} The signature of the action of the group above is $(2,4,2p^2)$  in the former case, and  $(2p^2, 2p^2,p^2)$ in the latter. Furthermore, for $l \neq l'$ in  $\{1, \ldots, p^2\}$ we have that 
$\mathscr{Z}_{p,l} \cong \mathscr{Z}_{p,l'}$  if and only if  $l+ll'+l' \equiv 0 \mbox{ mod }p^2,$ and there are exactly $\tfrac{p^2-2p+1}{2}$ isomorphism classes of Riemann surfaces $S$ in this case.
\end{theorem}

\begin{proof}
We consider the group $G\cong \mathbb{Z}_{2p^2}$ with the  presentation $G=\langle a,b \ |  \ a^{p^2}=b^2=[a,b]=1\rangle,$ and study each signature separately.

\s

{\bf (1)} Assume that $G$ acts on $S$ with signature $(2, p^2, 2p^2)$ and therefore $g=\tfrac{1}{2}(p^2-1).$ Since $G$ has only one involution and since, up to automorphisms of $G$, an element of order $p^2$ can be taken as $a$, one has that there is only one ske representing the action of $G$, up to $G$-equivalence. Thus, $S$ is uniquely determined. On the other hand, the curve  $\mathscr{X}_p$ has genus $\tfrac{1}{2}(p^2-1)$ and the map $(x,y) \mapsto (-x, \omega_{p^2}y)$ generates a group of automorphisms $\mathscr{G}$ of it which is isomorphic to $G$. The fact that the point $(0,1) \in \mathscr{X}_p$ has $\mathscr{G}$-stabiliser of order two shows that the signature of the action is necessarily equal to $(2, p^2, 2p^2).$ The uniqueness of $S$ asserts that the normalisation of $\mathscr{X}_p$ is isomorphic to $S$, as desired.

\s

{\bf (2)}  Assume that $G$ acts on $S$ with signature $(2p^2, 2p^2, p)$ and therefore $g=p(p-1).$ Write $$\Gamma=\langle x_1, x_2, x_3 \, | \,  x_1x_2x_3=x_1^{2p^2}=x_2^{2p^2}=x_3^p=1\rangle.$$ Up to automorphisms of $G$, an element of order $2p^2$ of $G$ can be taken as $ab.$ As the elements of $G$ of order $p$ are $a^{pm}$ for $m \in \{1, \ldots, p-1\}$ we see that, up to $G$-equivalence, each ske representing the action of $G$ is given by $$\Phi_m: \Gamma \to G \mbox{ defined by } \Phi_m(x_1, x_2, x_3)=(ab, a^{-1-pm}b, a^{pm})$$for some $m \in \{1, \ldots, p-1\}.$  In other words, if we write $S_m = \mathbb{H}/\mbox{ker}(\Phi_m)$ then $S \cong S_m$ for some $m \in \{1, \ldots, p-1\}.$ Notice, in addition, that for $m \neq m'$ the skes $\Phi_m$ and $\Phi_{m'}$ are $G$-inequivalent. By arguing as in the previous theorems one has that either the automorphism group of $S$ is isomorphic to $G$ or it has order $4p^2$ and acts with signature $(2, 2p^2, 2p).$ By \cite{BCC}, the latter case occurs if and only if there is an automorphism of $G$ fixing $a$ and sending $ab$ to $a^{-1-pm}b.$ It is easy to see that this automorphism exists if and only if $pm \equiv -2 \mbox{ mod } p^2$. Therefore   the automorphism group of $S$ is isomorphic to $G$. After noticing that $S$ is $p$-gonal, the determination of appropriate  rotation numbers shows that $S_m$ is isomorphic to the normalisation of   $\mathscr{Y}_{p,m}$ for each $m.$ The normaliser of $\Gamma$ is a Fuchsian group $\Gamma'$ of signature $(2,2p^2, 2p).$ If we write $$\Gamma'=\langle y_1, y_2, y_3 \, | \, y_1y_2y_3=y_1^2=y_2^{2p^2}=y_3^{2p}=1\rangle$$ and consider the inclusion $(x_1, x_2, x_3) \mapsto (y_3^{-1}y_2y_3, y_2, y_3^2)$ then $\Gamma'/\Gamma=\langle Y_3 \, | \, Y_3^2=1\rangle.$ 
Note that $Y_3\cdot \Phi_{m}=\Phi_{p-m}$ Thus, the isomorphism classes are represented by $S_m$ with $1 \leqslant m \leqslant \tfrac{p-1}{2},$ and the proof is done.

\s

{\bf (3)} Assume that $G$ acts on $S$ with signature $(2p, p^2, 2p^2)$ and therefore $g=\tfrac{(p-1)(2p+1)}{2}$. Write $$\Gamma=\langle x_1, x_2, x_3 \, | \,  x_1x_2x_3=x_1^{2p}=x_2^{p^2}=x_3^{2p^2}=1\rangle.$$ By reasoning as in the previous case, it can be seen that the number of $G$-equivalence classes of  skes representing the action of $G$ on $S$ is $p-1,$ and that they are represented by the skes $$\Phi_n: \Gamma \to G \mbox{ defined by } \Phi_n(x_1, x_2, x_3)=(a^{pn}b, a, a^{-1-pn}b)$$where $n \in \{1, \ldots, p-1\}.$ Thus, $S$ is isomorphic to $S_n =\mathbb{H}/ \mbox{ker}(\Phi_n)$ for some $n \in \{1, \ldots, p-1\}.$  The maximality of $\Gamma$ implies that the  automorphism group of $S$ is isomorphic to $G,$ and that $S_n$ and $S_{n'}$ are non-isomorphic for $n \neq n'$. On the other hand, note that the genus of  $\mathscr{W}_{p,k}$ agrees with the genus of $S.$ Moreover, $\mathscr{W}_{p,k}$ has a group of automorphisms $\mathscr{G} \cong \mathbb{Z}_{2p^2}$ generated by the map $$t: (x,y) \mapsto (-x, (-1)^k\omega_{p^2}y).$$  Observe that $ t^{2}$ is a $p^2$-gonal automorphism of $\mathscr{W}_{p,k}$ and that the corresponding  quotient map  $\mathscr{W}_{p,k} \to \mathscr{W}_{p,k}/\langle t^2 \rangle \cong \bar{\mathbb{C}}$ ramifies over four values: $0$ marked with $p$ and $\infty, \pm 1$ marked with $p^2.$ Thus, the signature of the action of $\mathscr{G}$ on $\mathscr{W}_{p,k}$ is $(2p, p^2, 2p^2).$ Consequently, each $S_n$ is isomorphic to the normalisation of $\mathscr{W}_{p,k}$ for some $k.$ Finally,  after identifying $\mathscr{W}_{p,k}/\mathscr{G}$ with the Riemann sphere with  $0, 1$ and $\infty$ marked with $2p, p^2$ and $2p^2$ respectively, each isomorphism between $\mathscr{W}_{p,k}$ and $\mathscr{W}_{p,k'}$  induces a M\"{o}bius transformation fixing $\infty, 0$ and $1.$  It then follows that  $\mathscr{W}_{p,k} \cong \mathscr{W}_{p,k'}$ if and only if $k=k'.$ All the above says that there is a bijective correspondence between  $$\{S_n \, | \, n \in \{1, \ldots, p-1\}\} \mbox{ and }\{\mathscr{W}_{p,k} \, | \, k \in \{1, \ldots, p-1\}\},$$and the proof is done.

\s

{\bf (4)} Assume that $G$ acts on $S$ with signature $(2p^2, 2p^2, p^2)$ and therefore $g=p^2-1$. If we write $$\Gamma=\langle x_1, x_2, x_3 \, | \,  x_1x_2x_3=x_1^{2p^2}=x_2^{2p^2}=x_3^{p^2}=1\rangle$$ then it can be seen that the  $G$-equivalence classes of skes representing the action of $G$ on $S$  are represented by  $$\Phi_l: \Gamma \to G \mbox{ defined by } \Phi_l(x_1, x_2, x_3)=(ab, a^{-l-1}b, a^{l})$$where $l \in \{1, \ldots, p^2\}$  satisfies  $(l,p^2)=(l+1,p^2)$=1. Thus, $S$ is isomorphic to $S_l =\mathbb{H}/ \mbox{ker}(\Phi_l)$ for some $l$ as before. Note that there is an automorphism of $G$ fixing $a$ and sending $ab$ to $a^{-l-1}b$ if and only if $l=p^2-2$. It follows from   \cite{BCC} that the automorphism group of $S_l$ is isomorphic to $G$ if and only if $l \neq p^2-2.$ To determine the  automorphism group of $S_{p^2-2}$, consider the group $\mathbb{Z}_{p^2} \rtimes \mathbb{D}_4$ with presentation \eqref{azul} and the ske $$\Theta : \Gamma''=\Gamma''_{(2,2p^2, 4)} \to \mathbb{Z}_{p^2} \rtimes \mathbb{D}_4 \mbox{ given by } \Theta(z_1, z_2, z_3)=(TR^2, A^{-1}TR, AR)$$  where $z_1, z_2$ and $z_3$ are canonical generators of $\Gamma''.$ A routine computation shows that the subgroup $H=\langle A, TR^3\rangle \cong \mathbb{Z}_{2p^2}$ of \eqref{azul} acts on $X:=\mathbb{H}/\mbox{ker}(\Theta)$ with signature $(p^2, 2p^2, 2p^2).$ Consequently, $X \cong S_{p^2-2}$ and the automorphism group of $S_{p^2-2}$ is isomorphic to \eqref{azul}. The normaliser of $\Gamma$ is a Fuchsian group $\Gamma'$ of signature $(2,2p^2, 2p^2).$ By proceeding analogously as in the proof of statement {\bf (2)}, the action by conjugation of $\Gamma'/\Gamma$ on $$\{\Phi_l: l \in \{1, \ldots, p^2\}, (l,p^2)=(l+1, p^2)=1 \mbox{ and } l \neq p^2-2 \}$$ gives rise to $\tfrac{p(p-2)-1}{2}$ orbits of length two only: $\{\Phi_l, \Phi_{l'}\}$ where $l'(1+l) \equiv -l \mbox{ mod }p.$ Thus, $S_l \cong S_{l'}$ if and only if $l+ll'+l'\equiv 0 \mbox{ mod }p^2,$ and therefore this case yields $1+\tfrac{p(p-2)-1}{2}=\tfrac{p^2-2p+1}{2}$ pairwise non-isomorphic Riemann surfaces.  Finally, observe that $S_l$ is a $p^2$-gonal Riemann surface and the corresponding $p^2$-gonal map is totally ramified. Thus, after determining the rotation numbers of $a$ at each one of its fixed points, one obtains that $S_l$ is isomorphic to the normalisation of the $\mathscr{Z}_{p,l}$ and this ends the proof.
\end{proof}

\begin{rema}
Observe that the Riemann surface appearing in case (1) in the theorem above is hyperelliptic, and none of the remaining Riemann surfaces in this paper enjoys this property.
\end{rema}


\section{Summary of the results}\label{s7}

The results presented in this paper provide a complete classification of compact Riemann surfaces admitting a triangular action of a group of order $2p^2$, where $p$ is an odd prime number.

\s

As a direct consequence of Theorems \ref{Tma1}, \ref{Tma2}, \ref{Tma3} and \ref{Tma4} we deduce that if $p\geqslant 5$ then there are precisely $\frac{p^2+2p+3}{2}$ pairwise  non-isomorphic  Riemann surfaces of genus at least two endowed  with a group of automorphisms of order $2p^2$ acting triangularly. If $p=3$ then such a number is 8. Moreover, as all the algebraic models determined in this paper are given in terms of affine curves given by polynomials with rational coefficients, we obtain the following result.

\begin{proposition}
    Each compact Riemann surface of genus at least two endowed  with a group of automorphisms of order $2p^2$ acting triangularly is defined over the rational numbers.
\end{proposition}
It is worth emphasising that the result above is very much in contrast with the situation for low genera, as shown in \cite{low}, and with the case of $pq$ automorphisms, as obtained in \cite{StreitWolfart}.

\s

We summarise the classification in the  table below. 
Here, we denote by $\mbox{sig.}G$ and $\mbox{sig.Aut}(S)$  the signature of the action of $G$ and of $\mbox{Aut}(S)$ respectively, and  $\#S$ is the number of pairwise non-isomorphic Riemann surfaces that admit an action with the indicated signature (in the first row we have that $p \neq 3$).

\s

\begin{center}
\renewcommand{\arraystretch}{1.5}

\begin{tabular}{|c|c|c|c|c|c|}
\hline
$G$ & sig.$G$ & $\#S$ & $\mbox{Aut}(S)$ & sig.$\mbox{Aut}(S)$ & genus \\
\hline \hline

\multirow{4}{*}{$\Bbb{D}_p\times \Bbb{Z}_p$}
& $(2,p,2p)$ & $1$ & $\Bbb{Z}_p^2\rtimes \Bbb{D}_3$ & $(2,3,2p)$ & $\frac{1}{2}(p-2)(p-1)$ \\ \cline{2-6}

& \multirow{3}{*}{$(2p,2p,p)$}
&  $1$ & $\Bbb{D}_p^2\rtimes \Bbb{Z}_2$ & $(2,2p,4)$ & $(p-1)^2$ \\ \cline{3-6}
&  & $1$ & $\Bbb{D}_p\times \Bbb{Z}_{2p}$ & $(2p,2p,2)$ & $(p-1)^2$\\ \cline{3-6}

&  & $\frac{p-3}{2}$ & $\Bbb{D}_p\times \Bbb{Z}_{p}$ & & $(p-1)^2$ \\ \hline

\multirow{6}{*}{$\Bbb{Z}_{2p^2}$}
& $(2,p^2,2p^2)$ &  $1$  & $ \Bbb{Z}_{2p^2}$  & & $\frac{1}{2}(p^2-1)$ \\ \cline{2-6}

& $(2p^2,2p^2,p)$ & $\frac{p-1}{2}$ & $ \Bbb{Z}_{2p^2}$  & & $p(p-1)$\\ \cline{2-6}

& $(2p,p^2,2p^2)$ & $p-1$ & $ \Bbb{Z}_{2p^2}$  & & $\frac{1}{2}(p-1)(2p+1)$ \\ \cline{2-6}

& \multirow{2}{*}{$(2p^2,2p^2,p^2)$}
& $\frac{p^2-2p-1}{2}$ & $ \Bbb{Z}_{2p^2}$  &  & $p^2-1$\\ \cline{3-6}

& &  $1$ & $\Bbb{Z}_{p^2}\rtimes \Bbb{D}_4$ & $(2,4,2p^2)$ & $p^2-1$\\ \cline{3-6}
 \hline

\end{tabular}
\end{center}

\s

\begin{rema}
As a straightforward consequence of our results we recover the strong symmetric genus for $\mathbb{D}_p \times \mathbb{Z}_p$ and $\mathbb{Z}_{2p^2}$, namely, the minimum genus among the Riemann surfaces on which these groups act as a group of automorphisms. See \cite{Harvey} and \cite{Patra}.
\end{rema}

\begin{rema} Let $S$ be a $p$-gonal Riemann surface of genus $g.$ According to the classical Severi-Castelnuovo inequality \cite{accola} if $g>(p-1)^2$ then the $p$-gonal group of $S$ is unique. Moreover, according to \cite[Theorem 1]{gabino}, if there are two $p$-gonal groups then they are conjugate in the automorphism group of $S.$ The automorphism groups of $p$-gonal Riemann surfaces admitting more than one $p$-gonal group were determined in \cite{Wootton2}. Observe that among our Riemann surfaces two of them have more than one $p$-gonal group: they correspond to the first two cases in the table above.
\end{rema}

Finally, the correspondence between quasiplatonic Riemann  surfaces and orientably-regular hypermaps allows us to obtain the following result.

\begin{proposition}
Let $p \geqslant 5$ be a prime number. There are precisely $p^2+p$ pairwise non-equivalent orientably-regular hypermaps of genus $g \geqslant 2$ with orientation-preserving automorphism group of order $2p^2$. They are classified as follows. 

\s

{\bf (1)} There are precisely $p$ orientably-regular hypermaps whose orientation-preserving automorphism group is isomorphic to $\mathbb{D}_p \times \mathbb{Z}_p.$ One of them is a reflexive map of type $\{p,2p\}$ while the remaining $p-1$ hypermaps are of type $\{2p,2p,p\}.$ Among the latter, two are supported on two distinct surfaces, whereas the remaining $p-3$ occur in pairs, namely, they are supported on $\tfrac{p-3}{2}$ different surfaces.  

\s

{\bf (2)} There is precisely one orientably-regular hypermap whose orientation-preserving automorphism group is isomorphic to $\mathbb{Z}_{2p}\times \mathbb{Z}_p.$ The  hypermap is of type is $\{p, 2p,2p\}$ and is supported on the same surface as one of the hypermaps described in item {\bf (1)}.

\s

{\bf (3)} There are precisely $p^2-1$ orientably-regular hypermaps whose orientation-preserving automorphism group is isomorphic to $ \mathbb{Z}_{2p^2}.$ One of them is a reflexive map of type $\{p^2, 2p^2\}$, $p-1$ of them are of type $\{2p^2, 2p^2, p\}$ and occur in pairs, and $p-1$ of them are of type $\{2p, p^2, 2p^2\}$. The remaining  $p^2-2p$ hypermaps are of type   $\{2p^2, 2p^2, p^2\}$ and $p^2-2p-1$ of them occur in pairs.  
\end{proposition}

{\bf Conflict of Interest and Data Availability.} The authors have no conflict of interest to declare that is relevant to this article. All data generated or analysed during this study are included in this manuscript.

\end{document}